\documentclass[11pt]{article}

\usepackage{amsmath,amssymb,amsthm,amscd,epsfig}
\newcommand{\vp}{\varepsilon}

\theoremstyle{plain}

\newtheorem{thm}{Theorem}

\newtheorem*{unthm}{Theorem}

\newtheorem{lem}{Lemma}
\newtheorem{cor}{Corollary}
\newtheorem*{uncor}{Corollary}

\newtheorem{pro}{Proposition}

\theoremstyle{definition}

\newtheorem{defn}{Definition}

\theoremstyle{remark}

\begin{document}

\title{Approximations for Gibbs states of arbitrary H\"{o}lder potentials on hyperbolic folded sets}

\author{Eugen Mihailescu}

\date{}
\maketitle

\begin{abstract}

In the case of smooth non-invertible maps which are hyperbolic on
folded basic sets $\Lambda$, we give approximations for the Gibbs
states (equilibrium measures) of arbitrary H\"{o}lder potentials,
with the help of weighted sums of atomic measures on preimage sets
of high order. Our endomorphism may have also stable directions on
$\Lambda$, thus it is non-expanding in general. Folding of the phase space means that we do not have a foliation structure for the local unstable manifolds
(they depend on the whole past and may intersect each other both inside and outside $\Lambda$), and moreover the number of
preimages remaining in $\Lambda$ may vary; also Markov partitions
do not always exist on $\Lambda$. Our convergence results apply also to Anosov endomorphisms, in
particular to Anosov maps on
 infranilmanifolds.

\end{abstract}



\section{Introduction}\label{sec1}

Gibbs states (equilibrium measures) of H\"{o}lder potentials for
smooth maps appear naturally in statistical mechanics (for example
\cite{Bo}, \cite{Ru-survey99}, \cite{S}, etc). In order to study
equilibrium measures for smooth diffeomorphisms, one can use the
specification property in order to find plenty of periodic points,
which in turn can be used to approximate equilibrium
 measures. This method was employed successfully by Bowen in \cite{Bo} (see also \cite{KH}).
 Also in \cite{Bo} it is studied the weighted distribution of preimages for one-sided shifts of finite type
 (which are examples of non-invertible maps).

The setting that we work with in this paper is \textit{different}
from both the hyperbolic diffeomorphism case (see \cite{Bo}), as
well as from the expanding case (see \cite{Ru}). We assume that
$f: M \to M$ is a smooth map (endomorphism), not necessarily a
diffeomorphism, and that $\Lambda$ is a basic set for $f$ so that
$f$ is hyperbolic on $\Lambda$, but not necessarily expanding. For
such \textit{non-expanding endomorphisms} there are no Markov
partitions in general, so it is not possible to code the system
using Markov partitions like in the diffeomorphism case. Also the
presence of stable directions on $\Lambda$ makes the local inverse
iterates of small balls to grow exponentially (up to a certain
order).
 For instance in \cite{Bot} it was introduced a family of horseshoes
with self intersections, and it was proved that there are open
sets of parameters which give non-injectivity of the map on the
respective basic set. \ In the non-invertible case, if $f$ is
hyperbolic on $\Lambda$, we do not have a foliation structure for
 the local unstable manifolds; the local unstable manifolds depend now on the whole prehistories (see \cite{Ru1}).
 This folding of the phase space is a major difference from the diffeomorphism
 case, since we are forced to work on $\hat \Lambda$ which is
 not a manifold.
 The local unstable manifolds may intersect each other and through any given point there may pass infinitely
 many local unstable manifolds. Moreover the number of $f$-preimages belonging to $\Lambda$ \textit{may vary} from
 point to point,  so the map is not necessarily constant-to-1; the set $\Lambda$ is \textbf{not} necessarily totally invariant.
 The local unstable manifolds depend Holder on their respective prehistories in the canonical metric on
 $\hat \Lambda$ (\cite{M}). \
The existence of several $n$-preimages in $\Lambda$ for any point
$x \in \Lambda$ means that we can have $n$-preimages $y \in
f^{-n}(x) \cap \Lambda$ of $x$ where the consecutive sum
$S_n\phi(y):= \phi(y) + \ldots + \phi(f^{n-1}y)$ is \textit{well
behaved}, but also other $n$-preimages $z \in f^{-n}(x) \cap
\Lambda$ where $S_n\phi(z)$ is \textit{badly behaved}.

 In the case of smooth non-invertible expanding
maps the situation is difficult, and the problem of finding the
weighted distributions of preimages was solved by Ruelle in
\cite{Ru}; in that situation it was important that the local
inverse iterates contract uniformly on small balls.

In our present non-invertible non-expanding setting, we will
describe the weighted distributions of preimages by studying the
intersections between different tubular neighbourhoods of the
(many) different local unstable manifolds, and by the use of
specification and the expression for equilibrium measure with the
help of periodic points. This will imply also the use of some
finer properties of the lifting of invariant measures to the
natural extension and the comparison between the different types
of behaviors of weighted sums of atomic measures on various
prehistories, with respect to Gibbs states. The \textbf{main
results} of the paper are in Theorem \ref{teorema} and its
Corollaries:

\begin{unthm}
Let $f:M \to M$ be a smooth map (say $\mathcal{C}^2$) on a smooth
Riemannian manifold $M$, so that $f$ is hyperbolic and
finite-to-one on a basic set $\Lambda$; assume also that the
critical set $\mathcal{C}_f$ of $f$ does not intersect $\Lambda$.
Let also $\phi$ a Holder continous potential on $\Lambda$ and
$\mu_\phi$ be the equilibrium measure of $\phi$ on $\Lambda$. Then
$$ \int_\Lambda |<\frac 1n \mathop{\sum}\limits_{y \in f^{-n}(x) \cap \Lambda}
\frac{e^{S_n\phi(y)}}{\mathop{\sum}\limits_{z \in f^{-n}(x) \cap \Lambda}
e^{S_n\phi(z)}} \cdot \mathop{\sum}\limits_{i=0}^{n-1} \delta_{f^i
y} - \mu_\phi, g >|  d\mu_\phi(x)
\mathop{\longrightarrow}\limits_{n \to \infty} 0, \forall g \in
\mathcal{C}(\Lambda, \mathbb R). $$
\end{unthm}

In Corollary \ref{subsequence} we obtain an approximation for the
equilibrium measure $\mu_\phi$, i.e the weak-$\star$ convergence
of a sequence of weighted atomic probabilities of the above type
towards $\mu_\phi$.

\begin{uncor}
In the same setting as in Theorem \ref{teorema}, for any Holder
potential $\phi$ with equilibrium measure $\mu_\phi$, it follows
that there exists a subset $E \subset \Lambda$, with $\mu_\phi(E)
= 1$ and an infinite subsequence $(n_k)_k$ such that for any $z
\in E$ we have the weak-$\star$ convergence of measures: $$\frac
{1}{n_k} \mathop{\sum}\limits_{y \in f^{-n_k}(x)\cap \Lambda}
\frac{e^{S_{n_k}\phi(y)}}{\mathop{\sum}\limits_{z \in f^{-n_k}(x) \cap \Lambda}
e^{S_{n_k}\phi(z)}} \cdot \mathop{\sum}\limits_{i=0}^{n_k-1}
\delta_{f^i y} \mathop{\longrightarrow}\limits_{k \to \infty}
\mu_\phi $$ In particular, if $\mu_0$ is the measure of maximal
entropy, it follows that for $\mu_0$-almost all
 points $x \in \Lambda$,
$
\frac {1}{n_k} \mathop{\sum}\limits_{y \in f^{-n_k} (x)\cap \Lambda}
\frac{\mathop{\sum}\limits_{i = 0}^{n_k-1} \delta_{f^i y}}
{\text{Card}(f^{-n_k} (x)\cap \Lambda)} \mathop{\longrightarrow}\limits_{k \to
\infty} \mu_0,$ for a subsequence $(n_k)_k$.
\end{uncor}

We remark that, since $\mu_\phi$ is positive on any open set (as any open set contains some small Bowen ball and one can apply Proposition 1), there exists a dense set in $\Lambda$ of points $x$ for which we have the above weak convergence of weighted atomic measures generated by $x$, towards $\mu_\phi$. Therefore in a physical non-reversible system, if one knows the past trajectories of such a generic point $x$ up to some high level $n$, then one can approximate the Gibbs state $\mu_\phi$ as above. \ 

One may also compare this result with the usual (forward) SRB measure
in the case of diffeomorphisms (see for example \cite{Bo},
\cite{Ru-survey99}, \cite{S}) or endomorphisms (\cite{L},
\cite{QS}, \cite{QZ}); in fact whenever $\mu_\phi$ is equivalent
to the Lebesgue measure, like in the case of toral endomorphisms
and $\phi \equiv 0$, we obtain an \textit{inverse SRB result}. Our
setting and methods are however different due to the lack of an
inverse function, the fact that unstable manifolds depend on whole
prehistories (not just base points), and also to the fact that the
number of preimages is not necessarily constant on $\Lambda$. \
We will apply the $L^1$ Ergodic Theorem of
von Neumann on the natural extension $\hat \Lambda$ for the lifted
measure, together with combinatorial arguments in order to
estimate the measure $\mu_\phi$ on the intersections between
different tubular unstable sets. Then we will estimate the
equilibrium measure on the different parts of the consecutive
preimage sets in $\Lambda$, by carefully
studying the prehistories from the point of view of 
convergence properties of certain weighted sums of Dirac measures
along them. \ In our Theorem let us notice that we average over
\textbf{all $n$-preimages} of points, so we do not consider only one
prehistory.
 This \textbf{simultaneous} consideration of all $n$-preimages is what makes the proof difficult.
 It cannot be obtained just by applying Birkhoff Ergodic Theorem
to different prehistories, since the speeds of convergence may be different over
 the \textit{uncountable}
collection of prehistories.

Among the \textbf{examples} of smooth endomorphisms on folded
basic sets, let us mention the horseshoes with self-intersections
from \cite{Bot}, the hyperbolic skew products with overlaps in
their fibers from \cite{MU}, or dynamical systems generating from
certain non-reversible statistical physics models (see
\cite{Ru-survey99}). Examples may be obtained also from
non-degenerate holomorphic maps on complex projective spaces (for
instance \cite{M-ETDS02}, \cite{M-Cam}).

Another application of Theorem \ref{teorema} will be in Corollary
\ref{Anosov}, where it will be applied to \textbf{Anosov
endomorphisms}, in order to give the distribution of consecutive
preimage sets, with respect to different equilibrium measures. A
classical example of Anosov endomorphism is given by a toral
endomorphism $f_A:\mathbb T^m \to \mathbb T^m, m \ge 2$, where
$f_A$ is the map induced on the $m$-dimensional torus by a matrix
$A$ with integer coefficients and $\text{det} A \ne 0$. Then any
point in $\mathbb T^m$ has exactly $|\text{det} A|$
$f_A$-preimages in $\mathbb T^m$, as can be seen since the $f_A$
image of the unit square is a parallelogram with area (volume)
equal to $|\text{det} A|$, whose corners have integer coordinates
(\cite{Wa}). If $A$ has all its eigenvalues of absolute values
different from 1, then $f_A$ is a hyperbolic endomorphism and the
above Theorem will apply. Since the equilibrium measure of any
constant function is the Haar measure (\cite{Wa}), we obtain the
asymptotic distribution of the local inverse iterates toward an
\textit{inverse SRB measure} in this case. \ A generalization of
this class of examples is given by smooth \textbf{perturbations of
hyperbolic toral endomorphisms} on $\mathbb T^m, m \ge 2$. They
will be again constant-to-1 and we will be able to apply our main
Theorem, to obtain the weighted distribution of preimages with
respect to equilibrium measures of Holder potentials.

\

\textbf{Remark:} Even on algebraic-type manifolds, like
\textbf{infranilmanifolds}, the situation of Anosov endomorphisms
is very different from that of Anosov diffeomorphisms. Indeed in
the case of infranilmanifolds (\cite{Ma}, \cite{Z}), Franks and
separately Manning showed that any Anosov diffeomorphism can be
"linearized", i.e it is topologically conjugate to some hyperbolic
automorphism. Also, Gromov (\cite{Ma}) showed that if $f$ is an
expanding map on a compact manifold, then $f$ is topologically
conjugate to some expanding endomorphism on some infranilmanifold.
However this is not the case for Anosov endomorphisms.  As was
proved in \cite{Z}, if $M$ is an infranilmanifold then there
exists a $\mathcal{C}^1$ dense subset $\mathcal{U}$ in the set of
"true" Anosov endomorphisms on $M$ (i.e those endomorphisms which
are not Anosov diffeomorphisms nor expanding maps), such that
every $f \in \mathcal{U}$ is not shift equivalent (hence also not
topologically equivalent) to any hyperbolic infranilmanifold
endomorphism. In particular this applies to tori $\mathbb T^m, m
\ge 2$, which are natural examples of infranilmanifolds. For such
Anosov endomorphisms which are neither diffeomorphisms nor
expanding, one cannot apply results similar to the ones from those
two previous cases. However we can apply Theorem \ref{teorema} to
get the distribution of preimage sets with respect to equilibrium
measures.

\section{Distributions of consecutive preimages on basic sets}\label{sec2}

First let us establish some notations. The next definition is
parallel to that of basic set  from \cite{KH} (see also
\cite{M-Cam}).

\begin{defn}\label{basic}
Let $f:M \to M$ a smooth map (say $\mathcal{C}^2$) defined on the
smooth manifold $M$. We will say that a compact $f$-invariant set
$\Lambda$ is a \textbf{basic set} for $f$ if there exists a
 neighbourhood $U$ of $\Lambda$ such that $\Lambda = \mathop{\cap}\limits_{n \in \mathbb Z} f^n(U)$, and
 if $f$ is transitive on $\Lambda$. \ As we work here with non-invertible maps, such sets will also be
called \textbf{folded basic sets}.
\end{defn}

\begin{defn}\label{hat}
The \textbf{natural extension} of the dynamical system $(f, \Lambda)$ is the dynamical system $(\hat f, \hat \Lambda)$,
where $\hat \Lambda:= \{\hat x = (x, x_{-1}, x_{-2}, \ldots), f(x_{-i})= x_{-i+1}, x_0 = x, x_{-i} \in \Lambda,
i \ge 1\}$ and $\hat f (\hat x) := (f(x), x, x_{-1}, \ldots), \hat x \in \hat \Lambda$. It follows that $\hat f$ is a
 homeomorphism on $\hat \Lambda$. An element $\hat x = (x, x_{-1}, \ldots)$ of $\hat \Lambda$, starting with $x$, is
 called a \textbf{prehistory} of $x$ (or \textit{full prehistory} of $x$). The canonical projection $\pi: \hat \Lambda
 \to \Lambda$ is defined by $\pi(\hat x) = x, \hat x \in \hat \Lambda$.
If $f(y) = x, y \in \Lambda$ we call $y$ a \textbf{preimage} of $x$; if $f^n(z) = x, z \in \Lambda$, we call $z$ an
$n$-\textbf{preimage} of $x$ (through $f$). A finite sequence $(x, x_{-1}, \ldots, x_{-n})$ will be called an
 $n$-\textit{prehistory} of $x$.
\end{defn}

Let us mention that if $\mu$ is an $f$-invariant probability measure on $\Lambda$, then
there exists a unique probability $\hat f$-invariant measure $\hat \mu$ on $\hat \Lambda$ such
that $\pi_*(\hat \mu) = \mu$. It can be seen that $\mu$ is ergodic if and only if $\hat \mu$ is
ergodic on $\hat \Lambda$. Also the topological pressure of $\phi$ (denoted by $P_f(\phi)$ to emphasize
dependence on $f$) is equal to the topological pressure of $\phi\circ \pi$, namely $P_{\hat f}(\phi\circ \pi)$;
and $\mu$ is an equilibrium measure for $\phi:\Lambda \to \mathbb R$ if and only if $\hat \mu$ is an equilibrium
measure for $\phi \circ \pi$.

The concept of hyperbolicity on $\Lambda$ makes sense for
non-invertible maps (i.e endomorphisms), but now the unstable
tangent subspaces and the local unstable manifolds depend on whole
prehistories, not only on the base points (for exp. \cite{Ru1}).
In this hyperbolic setting we will denote by $E^s_x, W^s_r(x)$ the
stable tangent subspace, respectively the local stable manifold at
$x$ (for $x \in \Lambda$); and by $E^u_{\hat x}, W^u_r(\hat x)$
the unstable tangent subspace, respectively the local unstable
manifold corresponding to the prehistory $\hat x$ (for $\hat x \in
\hat \Lambda)$. Also by $Df_s(x)$ we shall denote the stable
derivative $Df_x|_{E^s_x}$ and by $Df_u(\hat x)$ the unstable
derivative $Df_x|_{E^u_{\hat x}}$. In \cite{M} we studied the
Holder dependence of local unstable manifolds with respect to the
prehistories and proved a Bowen type formula, giving the unstable
dimension as being the zero $t^u$ of the pressure of the unstable
potential $\Phi^u(\hat x):= -\log |Df_u(\hat x)|, \hat x \in \hat
\Lambda$. We proved also that given a measurable partition of
$\hat \Lambda$ subordinated to the unstable manifolds, the
equilibrium measure of $\Phi^u$ has conditional measures which are
geometric of exponent $t^u$.

If $f:M \to M$ is a smooth map which is hyperbolic on a basic set
$\Lambda$, let $\hat x = (x, x_{-1}, \ldots, x_{-n}, \ldots) \in
\hat \Lambda, n \ge 1, \vp >0$ small. Then we call an $(n,
\vp)$-\textbf{tubular unstable neighbourhood} (or tubular unstable
set) the set $$T_n(\hat x, \vp):= \{ y \in \Lambda, \exists \
\text{an} \ n-\text{preimage} \ y_{-n} \in \Lambda \ \text{of} \
y, \text{s.t} \ d(f^i y_{-n}, f^i x_{-n}) < \vp, i = 0, \ldots,
n\}$$ The notion of tubular unstable set can be extended to those
$y \in M$ which have an $n$-preimage $y_{-n} \in M$ with the above
property, but since we will work in this paper only with measures
supported on $\Lambda$, we preferred to give the definition
restricted to $\Lambda$. It is important to keep in mind that
tubular unstable sets corresponding to two different prehistories
of the same point $x \in \Lambda$ \textit{may not be the same};
still they intersect in a set containing $x$.

It is well-known that any $f$-invariant measure $\mu$ on $\Lambda$ can be lifted to a unique $\hat f$-invariant
measure $\hat \mu$ on $\hat \Lambda$ such that $\pi_*(\hat \mu) = \mu$. It will be important to see exactly how to
 calculate the measure $\hat \mu$ of an arbitrary closed set from $\hat \Lambda$, in terms of the $\mu$-measures of
 sets in $\Lambda$.

\begin{lem}\label{measure-lift}

Let $f: \Lambda \to \Lambda$ be a continuous map on a compact metric space $\Lambda$, and $\mu$ an $f$-invariant
borelian probability measure on $\Lambda$. Let $\hat \mu$ be the unique $\hat f$-invariant probability measure on
$\hat \Lambda$ with the property that $\pi_*(\hat \mu) = \mu$. Then for an arbitrary closed set
$\hat E \subset \hat \Lambda$, we have that
$$
\hat \mu(\hat E) = \mathop{\lim}\limits_n \mu(\{x_{-n}, \exists \hat x = (x, \ldots, x_{-n}, \ldots) \in \hat E \})
$$
\end{lem}

\begin{proof}

The arbitrary closed set $\hat E$ is not necessarily of the form
$\pi^{-1}(E)$ for some $E$ borelian set in $\Lambda$. Let us
denote $\hat E_n:= \hat f^{-n}\hat E, n \ge 1$; then $\hat
\mu(\hat E_n) = \hat \mu(\hat E)$ since $\hat \mu$ is $\hat
f$-invariant. Let also $\hat F_n:= \pi^{-1}(\pi(\hat E_n)), n \ge
1$. We will prove that $$ \hat E = \mathop{\cap}\limits_{n\ge
0}\hat f^n(\hat F_n) $$ We have clearly $\hat E \subset \hat
f^n(\hat F_n), n \ge 0$. Let now a prehistory $\hat z \in
\mathop{\cap}\limits_{n \ge 0} \hat f^n\hat F_n$; then if $\hat z
= (z, z_{-1}, \ldots, z_{-n}, \ldots)$, we obtain that $z_{-n} \in
\pi\hat E_n, \forall n \ge 0$, hence $\hat z \in \hat E$ since
$\hat E$ is assumed closed. Thus we showed the equality $\hat E =
\mathop{\cap}\limits_{n\ge 0}\hat f^n(\hat F_n)$. Now let us
notice that the above intersection is decreasing, since $\hat
f^{n+1}\hat F_{n+1} \subset \hat f^n\hat F_n, n \ge 0$, since for
a prehistory from $\hat f^{n+1} \hat F_{n+1}$ the $(n+1)$-th entry
is from $\pi\hat E_{n+1}$ and the $n$-th entry is in $\pi \hat
E_n$, whereas the $(n+1)$-entry of a prehistory from $\hat f^n
\hat F_n$ can be \textbf{any} preimage of a point from $\pi\hat
E_n$. Since the above intersection is decreasing, we get
$$\hat\mu(\hat E) = \mathop{\lim}\limits_n\hat \mu(\hat f^n\hat
F_n) = \mathop{\lim}\limits_n\hat \mu(\hat F_n) =
\mathop{\lim}\limits_n \hat \mu(\pi^{-1}(\pi(\hat E_n))) =
\mathop{\lim}\limits_n \mu(\pi(\hat E_n)) = \mathop{\lim}\limits_n
\mu(\pi\circ \hat f^{-n} \hat E)$$  We used that $\pi_*\hat\mu =
\mu$ and that $\hat \mu$ is $\hat f$-invariant on $\hat \Lambda$.
\   Therefore we obtain that $\hat \mu(\hat E) =
\mathop{\lim}\limits_n \mu(\{x_{-n}, \exists \hat x \in \hat E,
\hat x = (x, \ldots, x_{-n}, \ldots)\})$.

\end{proof}

For a basic set $\Lambda$ for a smooth map $f$ we will denote by
$f_\Lambda^{-1}x, x \in \Lambda$ the set of $f$-preimages of $x$
\textbf{which belong to} $\Lambda$. Similarly $f_\Lambda^{-n}x$
will denote the $n$-preimages of $x$ belonging to $\Lambda$, i.e
$f_\Lambda^{-n}x := f^{-n}x \cap \Lambda$. In general in this
paper we will be interested only in the preimages belonging to
$\Lambda$. We will denote also by $S_n\phi(y)$ for a point $y \in
\Lambda$, the \textbf{consecutive sum} $S_n\phi(y):= \phi(y) +
\ldots + \phi(f^{n-1}y), n \ge 1$. Define then for $n \ge 1$ and
$x \in \Lambda$, the probability measure:

\begin{equation}\label{mu-n}
\mu_n^x:=  \frac 1n \mathop{\sum}\limits_{y \in f_\Lambda^{-n}x}
\frac{e^{S_n\phi(y)}}{\mathop{\sum}\limits_{z \in f_\Lambda^{-n}x}
e^{S_n\phi(z)}} \cdot \mathop{\sum}\limits_{i=0}^{n-1} \delta_{f^i
y}
\end{equation}

The $\mu_n^x, n \ge 1$ are probability measures and thus from the
weak compactness of the unit ball in the space of measures, we
obtain that any sequence of such measures (for $x\in \Lambda$
given) contains a convergent subsequence; the limit of such a
sequence $(\mu_{n_k}^x)_{k \ge 1}$ is then an $f$-invariant
probability measure on $\Lambda$. We will show that such limit
measures have in fact an important thermodynamical property,
namely they are equilibrium measures (i.e maximize in the
Variational Principle, \cite{Wa}).

We shall use in the sequel \textbf{equilibrium measures for
endomorphisms}; the existence of these measures for non-invertible
maps, and estimates on Bowen balls for these measures are similar
to the corresponding properties for diffeomorphisms (see the proof
of Proposition \ref{eq-endo} below). We denote the Bowen ball
$B_n(y, \vp):=\{z \in \Lambda, d(f^i y, f^i z) < \vp, i = 0,
\ldots, n-1\}$, for $y \in \Lambda, n \ge 1, \vp >0$.

\begin{pro}\label{eq-endo}
Let $\Lambda$ be a hyperbolic basic set for a smooth endomorphism $f:M \to M$, and $\phi$ be a Holder continuous
function on $\Lambda$. Then there exists a unique equilibrium measure $\mu_\phi$ for $\phi$ on $\Lambda$ which has
the following properties:

a) for any $\vp >0$ there exist positive constants $A_\vp, B_\vp$
and an integer $n_0\ge 1$ such that for any $y \in \Lambda, n \ge
n_0$, $$ A_\vp e^{S_n \phi(y) - nP(\phi)}  \le \mu_\phi(B_n(y,
\vp)) \le  B_\vp e^{S_n\phi(y)-n P(\phi)} $$

b) By working eventually with a finite iteration of $f$, we have
$$\mu_\phi = \mathop{\lim}\limits_{n \to \infty}
\frac{1}{P_\Lambda(f, \phi, n)} \mathop{\sum}\limits_{x \in
\text{Fix}(f^n)\cap \Lambda} e^{S_n\phi(x)} \delta_x,$$ where
$P_\Lambda(f, \phi, n) := \mathop{\sum}\limits_{x \in
\text{Fix}(f^n)\cap \Lambda} e^{S_n\phi(x)}, n \ge 1$.

\end{pro}
\begin{proof}

 a) We work in the natural extension $\hat \Lambda$ with the expansive homeomorphism
 $\hat f: \hat \Lambda \to \hat\Lambda$. The existence of a unique equilibrium measure for the Holder
 potential $\phi\circ \pi$ with respect to the expansive  homeomorphism $\hat f:\hat \Lambda \to \hat \Lambda$
 follows from the standard theory for homeomorphisms on compact metric spaces (see for example \cite{Bo}, \cite{KH});
 let us denote this equilibrium measure by $\hat \mu_\phi$. Then, given an $\hat f$-invariant
 probability measure $\hat \mu$ on $\hat \Lambda$, there exists a unique $f$-invariant measure $\mu$ on $\Lambda$ such
 that $\pi_*\hat\mu = \mu$. If we take the measure $\hat \mu_\phi$ instead of $\hat \mu$, we will obtain a measure
 $\mu_\phi$. It is easy to show that $h_{\hat \mu}(\hat f) = h_\mu(f)$ and that $P_{\hat f}(\phi\circ \pi) = P_f(\phi),
 \forall \phi \in \mathcal{C}(\Lambda, \mathbb R)$. Thus it follows that $\mu$ is an equilibrium measure
 for $\phi$ if and only if its unique $\hat f$-invariant lifting $\hat \mu$ is an equilibrium measure for $\phi \circ \pi$
 on $\hat \Lambda$. Thus $\mu_\phi:= \pi_* \hat \mu_\phi$ is the unique equilibrium measure for $\phi$ on
$\Lambda$.
Now we see that there exists a $k=k(\vp) \ge 1$ such that
 $\hat f^k(\pi^{-1}B_n(y, \vp)) \subset B_{n-k}(\hat f^k\hat y, 2\vp) \subset \hat \Lambda$, for any $y \in \Lambda$.
 On the other hand for any $\hat y \in \hat \Lambda$, we have $\pi(B_n(\hat y, \vp)) \subset B_n(y, \vp)$.
 These two inclusions and the $\hat f$-invariance of $\hat \mu_\phi$, together with the estimates for the
 $\hat \mu_\phi$-measure of the Bowen balls in $\hat \Lambda$ (from \cite{KH}) imply that there exist positive
 constants $A_\vp, B_\vp$ (depending on $\vp>0$ and on $\phi$) such
 that:
$$ A_\vp e^{S_n\phi(y)-nP(\phi)} \le \mu_\phi(B_n(y, \vp)) \le
B_\vp e^{S_n(\phi)(y)-nP(\phi)}, \forall y \in \Lambda, n \ge 1 $$
So the estimates for Bowen balls are true also for endomorphisms.

 b) The iterate of $f$ may be needed in order to have topological mixing (needed to guarantee specification, \cite{KH}). However without loss of generality we may assume that $f$ is topologically mixing on $\Lambda$.
If $x$ is a periodic point for $f|_\Lambda$, say $f^m(x) = x$, then we obtain a periodic point for $\hat f$, namely the prehistory $\hat x = (x, f^{m-1}(x), \ldots, f(x), x, \ldots, f(x), x, \ldots) \in \hat \Lambda$. Conversely, if $\hat x$ is a periodic point for $\hat f$, then $x$ is a periodic point (of the same period) for $f$. Similarly as for diffeomorphisms we prove that if $f$ is hyperbolic then $f$ satisfies specification. Then specification is used to show the convergence of the weighted sums of Dirac measures concentrated at periodic points towards $\mu_\phi$, in the same way as in \cite{KH}.

\end{proof}

Let us prove now a Lemma giving a relationship between the
measures of different parts of the preimage of some set of
positive measure.

\begin{lem}\label{inverse-it}
In the above setting, let a Holder potential $\phi:\Lambda \to
\mathbb R$ and its unique equilibrium measure $\mu_\phi$. Let us
consider $\vp >0$, $k$ disjoint Bowen balls $B_m(y_1, \vp), \ldots,
B_m(y_k, \vp)$ and a borelian set $A \subset f^m B_m(y_1, \vp)
\cap \ldots \cap f^m B_m(y_k, \vp)$ such that $\mu_\phi(A) > 0$;
denote by $A_1:=f^{-m}A \cap B_m(y_1, \vp), \ldots, A_k:=
f^{-m}A\cap B_m(y_k, \vp)$. Then there exists a positive constant
$C_\vp$ independent of $m, y_1, \ldots, y_m$ such that $$ \frac {1}{C_\vp}
\mu_\phi(A_j) \cdot \frac{e^{S_m \phi(y_i)}}{e^{S_m\phi(y_j)}} \le
\mu_\phi(A_i) \le C_\vp \mu_\phi(A_j) \cdot
\frac{e^{S_m\phi(y_i)}}{e^{S_m\phi(y_j)}}, i, j = 1, \ldots, m $$

\end{lem}

\begin{proof}

We shall denote the equilibrium measure $\mu_\phi$ by $\mu$ to
simplify notation, and will work with the restriction of $f$ to
$\Lambda$ (without saying this explicitly every time).

Similarly as in \cite{KH} or in \cite{Wa}, since the borelian sets
with boundaries of measure zero form a sufficient collection, we
may assume that each of the sets $A_i, A_j$ have boundaries of
$\mu$-measure zero. \  From construction $f^m(A_i) = f^m(A_j), i,
j = 1, \ldots, k$. But $\mu$ can be considered as the limit of the
sequence of measures $$\tilde \mu_n:= \frac {1}{P(f, \phi, n)}
\cdot \mathop{\sum}\limits_{x \in \text{Fix}(f^n)} e^{S_n\phi(x)}
\delta_x,$$ where $P(f, \phi, n):= \mathop{\sum}\limits_{x \in
\text{Fix}(f^n)} e^{S_n\phi(x)}, n \ge 1$.  So we have
\begin{equation}\label{ai}
\tilde \mu_n(A_i) = \frac {1}{P(f, \phi, n)} \cdot \mathop{\sum}\limits_{x \in \text{Fix}(f^n) \cap A_i} e^{S_n\phi(x)},
n \ge 1
\end{equation}

Let us consider now a periodic point $x \in \text{Fix}(f^n) \cap A_i$; by definition of $A_i$, it follows that $f^m(x) \in A_i$, so there exists a point $y \in A_j$ such that $f^m(y) = f^m(x)$. Of course the point $y$ does not have to be periodic necessarily. But we can use Specification Property (\cite{KH}, \cite{Bo}) on the hyperbolic compact locally maximal set $\Lambda$. Indeed if $\vp>0$ is fixed, then there exists a constant $M_\vp>0$ such that for all $n > M_\vp$, there is a $z \in \text{Fix}(f^n)$ s.t $z$ $\vp$-shadows the $(n-M_\vp)$-orbit of $y$. Thus $z \in B_m(y_i, 2\vp)$ if $y_i$ is so that $y \in B_m(y_i, \vp)$.

Let now $V$ be an arbitrary neighbourhood of the set $A_j$. Then
if $n$ is large enough and $\vp>0$ fixed, it follows that $z \in
V$. \  Let us also take two points $x, \tilde x \in
\text{Fix}(f^n)\cap A_i$ and assume the same periodic point $z\in
V \cap \text{Fix} (f^n)$ corresponds to both through the previous
procedure. This means that the $(n- M_\vp-m)$-orbit of $f^m z$
$\vp$-shadows the $(n-M_\vp-m)$-orbit of $f^m x$ and also the
$(n-M_\vp-m)$-orbit of $f^m \tilde x$. Hence the
$(n-M_\vp-m)$-orbit of $f^m x$ $2\vp$-shadows the
$(n-M_\vp-m)$-orbit of $f^m \tilde x$. But recall that we chose
$x, \tilde x \in A_i \subset B_m(y_i, \vp)$, hence $\tilde x \in
B_{n-M_\vp}(x, 2\vp)$.

Now we can split the set $B_{n-M_\vp}(x, 2\vp)$ in at most $N_\vp$
smaller Bowen ball of type $B_n(\zeta, 2\vp)$. In each of these
$(n, 2\vp)$-Bowen balls we may have at most one fixed point for
$f^n$. This holds since fixed points for $f^n$ are solutions to
the equation $f^n \xi = \xi$ and, on tangent spaces we have that
$Df^n - Id$ is a linear map without eigenvalues of absolute value
1. Thus if $d(f^i \xi, f^i \zeta) < 2\vp, i = 0, \ldots, n$ and if
$\vp$ is small enough, it follows that we can apply the Inverse
Function Theorem at each step. Therefore there will exist only one
fixed point for $f^n$ in the Bowen ball $B_n(\zeta, 2\vp)$. So
there may exist at most $N_\vp$ periodic points from
$\text{Fix}(f^n)\cap \Lambda$ which have the same point $z \in V$
attached to them by the above procedure. In conclusion, to each
point $x \in \text{Fix}(f^n)\cap A_i$, there corresponds at most
finitely many points $z\in V$ obtained by specification from the
above procedure; their number is smaller than $N_\vp$. Let us
notice also that if $x, \tilde x$ have the same point $z\in V$
attached to them, then as seen before $\tilde x \in B_{n-M_\vp}(x,
2\vp)$ and thus, from the Holder continuity of $\phi$:

$$|S_n\phi(x) - S_n\phi(\tilde x)| \le \tilde C_\vp,$$ for some
positive constant $\tilde C_\vp$ depending on $\phi$ (but
independent of $n, x$). This can be used then in the estimate for
$\tilde \mu_n(A_i)$, according to (\ref{ai}). Now we use the fact
that if $z \in B_{n-M_\vp}(y, \vp)$, then $f^m(z) \in
B_{n-M_\vp-m}(f^mx, \vp)$. From the Holder continuity of $\phi$
and the fact that $x \in A_i \subset B_m(y_i, \vp)$, it follows
that there exists a constant $\tilde C_\vp$ (we denote it the same
as before, without loss of generality)
 satisfying:
\begin{equation}
|S_n \phi(z) - S_n\phi(x)| \le |S_m\phi(y_i) - S_m\phi(y_j)| +
\tilde C_\vp,
\end{equation}
 for $n > n(\vp, m)$.
This, together with (\ref{ai}) and the fact that there are at most
$N_\vp$ points $x \in \text{Fix}(f^n)$ which have the same
attached $z \in V \cap \text{Fix}(f^n)$, imply that there exists a
constant $C_\vp>0$ such that
\begin{equation}\label{ineq1}
\tilde \mu_n(A_i) \le C_\vp \tilde \mu_n(V) \cdot
\frac{e^{S_m\phi(y_i)}}{e^{S_m\phi(y_j)}},
\end{equation}

where we recall that $A_i \subset B_m(y_i, \vp), A_j \subset B_m(y_j, \vp)$.
But since $\partial A_i, \partial A_j$ have $\mu$-measure zero, we obtain:
$$
\mu(A_i) \le C_\vp \mu(V) \frac{e^{S_m\phi(y_i)}}{e^{S_m\phi(y_j)}}, i, j =1, \ldots, k
$$

Now $V$ was chosen arbitrarily as a neighbourhood of $A_j$, thus
$$ \mu(A_i) \le C_\vp \mu(A_j)
\frac{e^{S_m\phi(y_i)}}{e^{S_m\phi(y_j)}}, i, j = 1, \ldots, k
$$
Similarly we prove also the other inequality.

 \end{proof}

\begin{thm}\label{teorema}
Let $f:M \to M$ be a smooth (say $\mathcal{C}^2$) map on a
Riemannian manifold $M$, which is hyperbolic and finite-to-one on
a basic set $\Lambda$ so that $\mathcal{C}_f \cap \Lambda =
\emptyset$. Assume that $\phi$ is a Holder continuous potential on
$\Lambda$ and that $\mu_\phi$ is the equilibrium measure of $\phi$
on $\Lambda$. Then with the notation from (\ref{mu-n}), $$ \int
_\Lambda |<\mu_n^x-\mu_\phi, g>| d\mu_\phi(x) \mathop{\to}\limits_{n
\to \infty} 0, \forall g \in \mathcal{C}(\Lambda, \mathbb R) $$
\end{thm}

\begin{proof}

We make the convention that all the preimages that we work with
are in $\Lambda$. So we shall write $f^{-n} x$ for $f_\Lambda^{-n}
x, n \ge 1, x \in \Lambda$.

If $\phi$ is a Holder continuous function on $\Lambda$ if follows
from Proposition \ref{eq-endo} that there exists a unique
equilibrium measure $\mu_\phi$ for $\phi$, and $\mu_\phi$ is the
push-forward of the equilibrium measure $\hat \mu_\phi$ of $\phi
\circ \pi$ on $\hat \Lambda$. For simplicity of notation, we shall
denote the measure $\mu_\phi$ by $\mu$, with $\phi$ being fixed.
This measure is ergodic as being an equilibrium measure.

Let us fix now a continuous test function $g : \Lambda \to \mathbb
R$. From von Neumann's $L^1$ Ergodic Theorem applied to the
homeomorphism $\hat f^{-1}:\hat \Lambda \to \hat \Lambda$ and the
potential $g\circ \pi$ , we know that

\begin{equation}\label{neumann}
\int_{\hat \Lambda} |\frac 1n \mathop{\sum}\limits_{i=0}^{n-1}
g(x_{-i}) - \int_{\hat \Lambda} g\circ \pi \ d\hat\mu| \
d\hat\mu(\hat x) \mathop{\to}\limits_{n\to \infty} 0,
\end{equation}
where the prehistory $\hat x = (x, x_{-1}, x_{-2}, \ldots, x_{-i},
\ldots) \in \hat \Lambda$.  Denote by $$ \Sigma_n(g, y):=
\frac{\mathop{\sum}\limits_{i=0}^{n-1} g(f^i y)}{n} - \int _\Lambda g
d\mu, y \in \Lambda, n \ge 2 $$

Hence for an arbitrary small $\eta >0$, we have from
(\ref{neumann}) that:

$$ \hat \mu(\{\hat x = (x, x_{-1}, \ldots)\in \hat \Lambda,
|\Sigma_n(g, x_{-n})|\ge \eta\}) \mathop{\to}\limits_{n\to \infty}
0 $$

So for any $\vp'>0, \vp' = \vp'(\eta) << \eta$, there exists
$n(\eta)\ge 1$ so that if $n
> n(\eta)$ then
 $$ \hat \mu(\hat x, |\Sigma_n(g, x_{-n})| \ge \eta) <
\vp' $$ But now $\{\hat x \in \hat \Lambda, |\Sigma_n(g, x_{-n})|
\ge \eta\}$ is a closed set in $\hat \Lambda$, thus we can apply
Lemma \ref{measure-lift} to prove that:

\begin{equation}\label{preh}
\mu(x_{-n} \in \Lambda, |\Sigma_n(g, x_{-n})| \ge \eta) < \vp',
\end{equation}
if $n$ is large enough ( without loss of generality we can take $n
> n(\eta)$).

Now let us consider a small $\vp>0$ with $\vp < \vp(\eta) << \eta$
such that $\omega(3\vp) < \eta$, where $\omega(r)$ denotes in
general the maximal oscillation of $g$ on a ball of radius $r>0$.
Let us take also a maximal set of mutually disjoint $n$-Bowen
balls $B_n(y, \vp)$ in $\Lambda$; denote the set of such $y$ by
$\mathcal{F}_n$. Thus $\{B_n(y, \vp), y \in \mathcal{F}_n\}$ is
our maximal set. If $z \notin \mathcal{F}_n$, then from the
definition, $B_n(z, \vp)$ must intersect some Bowen ball $B_n(y,
\vp), y \in \mathcal{F}_n$. Thus $B_n(z, \vp) \subset B_n(y,
3\vp)$. \  Let us notice also that if $w \in B_n(z, 3\vp)$ then
\begin{equation}\label{osc}
|\Sigma_n(g, w)| \le |\Sigma_n(g, z)| + \omega(3\vp)
\end{equation}

 In the sequel we will split different subsets of $\mathcal{F}_n$ in two disjoint subsets
$\mathcal{R}_n, \mathcal{G}_n$, with $\mathcal{R}_n \subset \{x
\in \Lambda, |\Sigma_n(g, x)| \ge 2\eta\}$ and $\mathcal{G}_n
\subset \{x \in \Lambda, |\Sigma_n(g, x)| < 2\eta\}$. Intuitively
$\mathcal{R}_n$ consists of the "bad" $n$-preimages (corresponding
to $g, \eta$) and $\mathcal{G}_n$ are the "good" $n$-preimages. 

Recall now that we denoted $f_\Lambda^{-n} x:= \Lambda \cap f^{-n}x, n \ge 1, x \in \Lambda$, and that $\Lambda$ is not necessarily totally invariant. 
Consider then $$I_n(g, x):= \mathop{\sum}\limits_{y \in f_\Lambda^{-n}x}
\frac{e^{S_n\phi(y)}}{\mathop{\sum}\limits_{z \in
f_\Lambda^{-n}x}e^{S_n\phi(z)}} \cdot |\Sigma_n(g, y)|, x \in \Lambda, n
\ge 1 $$

Denote also by $$ \Lambda(n, \eta):= \{ x\in \Lambda, \text{s. t} \ x \
\text{has at least one} \ n-\text{preimage} \ x_{-n} \ \text{with}
\ \Sigma_n(g, x_{-n})| \ge 2\eta\} $$

The problem is that the unstable manifolds may depend on the whole
prehistories, thus if we take $z, w \in f^{-n}x$, then $f^n B_n(w,
\vp)$ and $f^nB_n(z, \vp)$ may be different unstable tubular sets;
these unstable tubular sets intersect each other in a (possibly
smaller) set containing $x$. \ By taking $n$-preimages for all
points in $\mathcal{F}_n$ and then the corresponding tubular
unstable sets as above, we shall obtain a collection of such
tubular unstable sets which intersect each other in different
smaller pieces, denoted generically by $D$. \ We have to estimate
$\int_\Lambda I_n(g, x) d\mu(x)$, which is a sum of integrals on
sets of type $D$. The problematic terms in this sum are those of
type $\frac{e^{S_n\phi(y)}}{\mathop{\sum}\limits_{z \in f_\Lambda^{-n}x}
e^{S_n\phi(z)}} \cdot |\Sigma_n(g, y)| \mu(D)$, with $|\Sigma_n(g,
y)| \ge 2\eta$, but so that \textbf{at the same time} there exist
also good $n$-preimages for every point $x \in D$ . In other words
$x \in \Lambda(n, \eta)$, but $x$ has also good $n$-preimages. The
set of points where \textbf{all} $n$-preimages from $\Lambda$ are
bad, is easier to measure, by using Lemma \ref{measure-lift}. In
the integral $\int_\Lambda I_n(g, x) d\mu(x)$ we have to deal both
with the measures of subsets $D$, and with the $n$-preimages $y$
of points $x \in D$, namely with the quantities $\Sigma_n(g, y)|$.

For the measures of subsets $D$ we shall use the $f$-invariance of
$\mu$, namely $\mu(D) = \mu(f^{-n}(D))$. Then we will use the
estimate on the measure of the set of bad preimages from Lemma
\ref{measure-lift}, coupled with the existent control on the good
preimages given by $|\Sigma_n(g, y)| < \eta$. For the rest of the
preimages, the idea is to control the sum of the measures of these
inverse iterates by the measure of the set of bad preimages.

First let us notice that if $x \in \Lambda$ and $y, z \in f^{-1} x
\cap \Lambda, y \ne z$, then since the critical set
$\mathcal{C}_f$ does not intersect $\Lambda$, it must exist a
positive constant $\vp_0$ so that $d(y, z) > \vp_0$. Hence if
similarly for some $n \ge 1$, we take two distinct $n$-preimages
$y, z \in \Lambda$ of $x$ ($f^ny = f^n z = x, y \ne z$), then we
cannot have $y, z \in B_n(w, 3\vp)$ for any $w \in \Lambda$, if
$\vp << \vp_0$. \  We know also that the Bowen balls $B_n(y,
3\vp)$ cover $\Lambda$ when $y$ ranges in $\mathcal{F}_n$, and
$B_n(y_i, \vp) \cap B_n(y_j, \vp) = \emptyset$ for any two
different points $y_i, y_j$ from $\mathcal{F}_n$.

Let us take then a small $0< \beta < 1$ such that $\eta < \beta$;
to fix ideas we will consider $\beta = 3 \eta$. If $x \in
\Lambda$, denote by $R_n(x)$ the set of $n$-preimages $y \in \Lambda$ of $x$
with $|\Sigma_n(g, y)| > 2\eta$ (in fact $R_n(x)$ depends on the
$\eta$ as well, but we do not record this here in order to
simplify notation). Let us denote now

\begin{equation}\label{D}
D_n(\beta, \eta):= \{x \in \Lambda(n, \eta),
\frac{\mathop{\sum}\limits_{y \in R_n(x)}
e^{S_n\phi(y)}}{\mathop{\sum}\limits_{y \in f^{-n} x}
e^{S_n\phi(y)}} < \beta \}
\end{equation}

Now if $||g||:= \mathop{\sup}\limits_{y \in \Lambda} |g(y)|$, it
follows from definition that for any $y \in \Lambda$, we have
$|\Sigma_n(g, y)| \le 2 ||g||$; thus for a point $x \in D_n(\beta,
\eta)$:

$$ \aligned  \mathop{\sum}\limits_{y \in f_\Lambda^{-n}x} \frac{e^{S_n
\phi(y)}}{\mathop{\sum}\limits_{z \in f_\Lambda^{-n}x}e^{S_n\phi(z)}}
|\Sigma_n(g, y)|  = & \mathop{\sum}\limits_{y \in R_n(x)}
\frac{e^{S_n\phi(y)}}{\mathop{\sum}\limits_{z \in f_\Lambda^{-n}x}
e^{S_n\phi(z)}} |\Sigma_n(g, y)| + \mathop{\sum}\limits_{y \in
f_\Lambda^{-n}x \setminus R_n(x)}
\frac{e^{S_n\phi(y)}}{\mathop{\sum}\limits_{z \in
f_\Lambda^{-n}x}e^{S_n\phi(z)}} |\Sigma_n(g, y)|  \\ & \le 2||g|| \cdot
\beta + \eta
\endaligned
$$

Therefore
\begin{equation}\label{est-D}
\int_{D_n(\beta, \eta)} I_n(g, x) d\mu(x) \le 2||g|| \beta + 2\eta
\end{equation}

We will now restrict to the complement of $D_n(\beta, \eta)$ in
$\Lambda$, so we work with points $x \in \Lambda$ for which
$\frac{\mathop{\sum}\limits_{y \in R_n(x)}
e^{S_n\phi(y)}}{\mathop{\sum}\limits_{y \in f^{-n}x}
e^{S_n\phi(y)}} \ge \beta$.  From Proposition \ref{eq-endo} we
know that for any $z \in \Lambda, n \ge 1$, $A_\vp\cdot
e^{S_n\phi(z) - nP(\phi)} \le \mu(B_n(z, \vp)) \le B_\vp \cdot
e^{S_n\phi(z) - nP(\phi)}$. Therefore we have

\begin{equation}\label{JN}
\mathop{\sum}\limits_{y \in f_\Lambda^{-n} x} \mu(B_n(y, \vp)) \le
\frac{C_\vp}{\beta} \cdot \mathop{\sum}\limits_{z \in R_n(x)}
\mu(B_n(z, \vp))
\end{equation}

Obviously from the definition of $\mathcal{F}_n$ we can place any
$n$-preimage of $x$ in a distinct Bowen ball of type $B_n(z,
3\vp)$, for some $z \in \mathcal{F}_n$. From Proposition
\ref{eq-endo} it follows that the $\mu$-measure of a Bowen ball of
type $B_n(z, \vp)$ is comparable to the $\mu$-measure of the Bowen
ball $B_n(z, 3\vp)$ (by \textbf{comparable} we mean that their
quotient is bounded below and above by a positive constant
independent of $z, n$). \ The idea is to split now the set
$\Lambda \setminus D_n(\beta, \eta)$ into subsets $A$ such that
points from $A$ have the same number of $n$-preimages. Let us
assume that $d$ is the maximum number of $1$-preimages that a
point may have in $\Lambda$. Then any point from $\Lambda$ has at
most $d^n$ $n$-preimages in $\Lambda$. We will use now tubular
unstable sets $T_n(\hat x, 3\vp)$ obtained from different Bowen
balls centered at the points of $\mathcal{F}_n$.  For an integer
$\ell \ge 2$ let us consider the sets of type

$$f^n(B_{i_1}) \cap \ldots \cap f^n(B_{i_\ell}) \setminus
[\mathop{\cup}\limits_{\ell < |J| \le d^n, J \subset
\mathcal{F}_n} \mathop{\cap}\limits_{j \in J} f^n(B_j)],$$

where $B_{j}:= B_n(y_j, 3\vp)$ for some points $y_j \in
\mathcal{F}_n$, and where we assume no repetitions among the
respective Bowen balls. This means actually that we do not repeat
a Bowen ball $B_j$ in the above intersection; if $y, z$ are
different $n$-preimages of the same point $x$ and $\vp$ is small
enough, then we cannot have $y, z$ in the same $B_n(\zeta, 3\vp)$.
Each point in such a set has exactly $\ell$ $n$-preimages, one in
each of the Bowen balls $B_{i_1}, \ldots, B_{i_\ell}$.

Let us denote the collection of all such sets by $F(\ell, n, \vp)$
and let $D \in F(\ell, n, \vp)$.
 Consider the sets of type $D \setminus \mathop{\cup}\limits_{D'
 \in F(\ell, n, \vp), D' \ne D} D', D \in F(\ell, n, \vp)$.
 These sets are now mutually disjointed, borelian, and cover
 $\Lambda$. Their collection will be denoted by $\tilde F(\ell, n,
 \vp)$.

We want to estimate the measure $\mu(\Lambda \setminus D_n(\beta,
\eta))$. In order to do this, we will split $\Lambda \setminus
D_n(\beta, \eta)$ into mutually disjoint borelian subsets,
obtained by intersecting the sets of $\tilde F(\ell, n, \vp)$ with
$\Lambda \setminus D_n(\beta, \eta)$. Let us denote the collection
of these intersections by $H(\ell, n, \vp), \ell \ge 1$.

We will take an arbitrary subset $S \in H(\ell, n , \vp)$, say $S
\subset f^n(B_{i_1}) \cap \ldots \cap f^n(B_{i_\ell})$. From the
$f$-invariance of $\mu$, it follows that $\mu(S) = \mu(f^{-n} S)=
\mu(f^{-n} S \cap B_{i_1}) + \ldots + \mu(f^{-n}S \cap
B_{i_\ell})$. We have that the Bowen balls $B_{i_s}, s = 1,
\ldots, \ell$ are mutually disjoint as they contain different
$n$-preimages of the same point (we know that for small $\vp$ one
cannot have two different $n$-preimages of the same point,
belonging to the same $B_n(y, 3\vp), y \in \Lambda$). Let us
denote by $S_n(i_1):= f^{-n}S \cap B_{i_1}, \ldots, S_n(i_k):=
f^{-n} S \cap B_{i_k}$. So $$\mu(S) = \mu(S_n(i_1) + \ldots +
\mu(S_n(i_k))$$ We assume that $\mu(S_n(i_1)) >0$, otherwise we
can take a different $S_n(j)$. But now from Lemma \ref{inverse-it}
and the fact that $S \subset \Lambda \setminus D_n(\beta, \eta)$,
it follows that

\begin{equation}\label{co}
\aligned & \mu(S_n(i_1)) + \ldots + \mu(S_n(i_k)) \le
C_\vp\frac{\mu(S_n(i_1))}{\mu(B_{i_1})} \left[ \mu(B_{i_1}) +
\ldots + \mu(B_{i_k})\right] \le \\ & \le
C_\vp\frac{\mu(S_n(i_1))}{\mu(B_{i_1})} \cdot \frac {1}{\beta}
\mathop{\sum} \limits_{B_j \in R_n(x), j \in \{i_1, \ldots, i_k\}}
\mu(B_j)
\endaligned
\end{equation}

Suppose that $T$ is another disjoint set from some $H(p, n, \vp)$,
so that $T \subset f^n(B_{i_1}) \cap f^n(B_{j_2}) \cap \ldots \cap
f^n(B_{j_p})$ and let $T_n(i_1), T_n(j_2), \ldots, T_n(j_p)$ be
the corresponding parts of $f^{-n}T$ belonging respectively to
$B_{i_1}, B_{j_2}, \ldots, B_{j_p}$. So the sets $S_n(i_1),
\ldots, S_n(i_\ell), T_n(i_1), T_n(j_2), \ldots, T_n(j_p)$ are
mutually disjoint borelian subsets. We assumed that $S$ and $T$
have both $n$-preimages in $B_{i_1}$ (for example). If they have
$n$-preimages in completely different Bowen balls, then the
situation will be simpler, since there will be no repetitions
below in (\ref{ST}).

Let us estimate now $\mu(S) + \mu(T)$; for this we consider two
points $x \in S$ and $y \in T$ and assume that $\{i_1, \ldots,
i_\ell\} \cap \{i_1, j_2, \ldots, j_p\} \cap R_n(x) = \{l_1,
\ldots, l_r\}$. Then from (\ref{co}), we obtain:

\begin{equation}\label{ST}
\aligned  \mu(S) + \mu(T) & \le \frac{C_\vp}{\beta}  \left[
\frac{\mu(S_n(i_1))}{\mu(B_{i_1})} \cdot \mathop{\sum}\limits_{B_j
\in R_n(x), j \in \{i_1, \ldots, i_\ell\}} \mu(B_j) +
\frac{\mu(T_n(t_1))}{\mu(B_{i_1})} \cdot \mathop{\sum}\limits_{B_s
\in R_n(y), s \in \{i_1, j_2, \ldots, j_p\}} \mu(B_s) \right]
\\ & \le  \frac{C_\vp}{\beta \mu(B_{i_1})} \cdot \left[\mu(S_n(i_1)) +
\mu(T_n(i_1)) \right] \cdot \left[\mu(B_{l_1}) + \ldots
\mu(B_{l_r})\right] + \\ &     \hspace{0.6in} + C_\vp \cdot
\frac{\mu(S_n(i_1))}{\beta \mu(B_{i_1})} \cdot \Sigma(S, n) +
C_\vp \cdot \frac{\mu(T_n(i_1))}{\beta \mu(B_{i_1})} \cdot
\Sigma(T, n),
\endaligned
\end{equation}
where $$\Sigma(S, n):= \mathop{\sum}\limits_{j \in
R_n(x)\cap\{i_1, \ldots, i_\ell\} \setminus \{l_1, \ldots, l_r\} }
\mu(B_j)  \ \ \text{and} \  \ \Sigma(T, n):=
\mathop{\sum}\limits_{j' \in \{i_1, j_2, \ldots, j_p\} \cap R_n(x)
\setminus \{l_1, \ldots, l_r\}} \mu(B_{j'})$$

But recall that the subsets $S_n(i_1)$ and $T_n(i_1)$ are disjoint
inside $B_{i_1}$, hence $\mu(S_n(i_1)) + \mu(T_n(i_1)) \le
\mu(B_{i_1})$. Since $\{l_1, \ldots, l_r\} = \{i_1, \ldots, i_l\}
\cap \{i_1, j_2, \ldots, j_p\} \cap R_n(x)$, the sums $\Sigma(S,
n)$ and $\Sigma(T, n)$ do not have common terms and do not have
any term from the collection $\{B_{l_1}, \ldots, B_{l_r}\}$. Thus
from (\ref{ST}) we obtain that $$ \mu(S) + \mu(T) \le \frac
{C_\vp}{\beta} \left[\mu(B_{l_1}) + \ldots + \mu(B_{l_r}) +
\Sigma(S, n) + \Sigma(T, n)\right] $$

From (\ref{osc}) we obtain that $B_i \subset \{y \in \Lambda,
|\Sigma_n(g, y)| \ge \eta\}$. Also from the estimates of
Proposition 1, we know that there exists a positive constant
$\chi_\vp$ such that $\mu(B_i) \le \chi_\vp \mu(B_n(y_i, \vp)),
i\ge 1$. Therefore, recalling also that the balls $B_n(y_i, \vp),
y_i \in \mathcal{F}_n$ are mutually disjoint we have that

$$ \aligned  &\mathop{\sum}\limits_{1 \le k \le r}\mu(B_{l_k}) +
\Sigma(S, n) + \Sigma(T, n) \le \chi_\vp \cdot
[\mathop{\sum}\limits_{1 \le k \le r} \mu(B_n(y_{l_k}, \vp)) +
\mathop{\sum}\limits_{j \in R_n(x) \cap \{i_1, \ldots, i_\ell\}
\setminus \{l_1, \ldots, l_r\}} \mu(B_n(y_j, \vp)) + \\
 & +\mathop{\sum}\limits_{j' \in R_n(x) \cap \{i_1, j_2, \ldots, j_p\}
\setminus \{l_1, \ldots, l_r\}} \mu(B_n(y_{j'}, \vp))]  = \chi_\vp
\cdot \mu(\mathop{\cup}\limits_{j \in R_n(x)} B_n(y_j, \vp)) \le
\chi_\vp \cdot \mu(y, |\Sigma_n(g, y)| \ge \eta)
\endaligned$$
 But from (\ref{preh}) we can control
the total measure of the set of bad $n$-preimages, which is thus
smaller than $\vp'$. Hence
\begin{equation}\label{mollie}
\mathop{\sum}\limits_{1 \le k \le r}\mu(B_{l_k}) + \Sigma(S, n) +
\Sigma(T, n) \le \chi_\vp \cdot \vp'
\end{equation}

This procedure can be used for any collection of mutually disjoint
borelian subsets from the collections $H(\ell, n, \vp), 1 \le \ell
\le d^n$, not only for $S, T$. Indeed by using Lemma
\ref{inverse-it} and the disjointness of sets from $H(\ell, n,
\vp)$ (hence also the mutual disjointness of the sets of their
$n$-preimages) we see that the weights associated in (\ref{ST}) to
any measure $\mu(B_j)$ (where $B_j$ corresponds to a bad
$n$-preimage) never add up to more than 1.

Then similarly as in (\ref{mollie}), by employing the control on
the total measure of bad $n$-preimages from (\ref{preh}), we can
conclude that for $n > n(\eta)$:

\begin{equation}\label{conc}
\mu(\Lambda \setminus D_n(\beta, \eta)) \le C_\vp
\cdot\frac{\vp'}{\beta}\cdot \chi_\vp = \tilde C_\vp \cdot
\frac{\vp'}{\beta}
\end{equation}

Therefore by using (\ref{est-D}) and (\ref{conc})

$$ \aligned
 \int_\Lambda I_n(g, x) d\mu(x) & \le 2\eta +
\int_{\Lambda(\eta)} I_n(g, x) d\mu(x) \le \\
& \le 2\eta + \int_{D_n(\beta, \eta)} I_n(g, x) d\mu(x) + \int_{\Lambda
\setminus D_n(\beta, \eta)} I_n(g, x) d\mu(x)  \le \\ &  \le
2\eta + 2||g|| \beta + 2\eta + 2 ||g|| \mu(\Lambda \setminus
D_n(\beta, \eta)) \le 4\eta + 2||g|| \beta + 2||g|| \tilde
C_\vp\frac{\vp'}{\beta},
\endaligned
$$ for $n > n(\eta)$. Recall however that we assumed before that
$3\eta= \beta$. Assume that $\vp'$ is so small that $\tilde C_\vp
\cdot \frac{\vp'}{3\eta} < \eta$.
 Then from the
last displayed inequality, it follows that there exists a positive
constant $C' = 4 + 8||g||$ so that:

$$ \int_\Lambda I_n(g, x) d\mu(x) \le C' \cdot \eta, \text{for} \
n
> n(\eta) $$

This shows in conclusion that $$ \int_\Lambda I_n(g, x) d\mu(x)
\mathop{\to}\limits_{n \to \infty} 0, \forall g \in
\mathcal{C}(\Lambda, \mathbb R). $$
\end{proof}

Hence we proved the convergence in integral (with respect to
$d\mu_n(x)$) of the measures $\mu_n^x$ from (\ref{mu-n}), towards
the equilibrium measure $\mu_\phi$ of $\phi$, in the hyperbolic
non-invertible case.

\begin{cor}\label{subsequence}
In the same setting as in Theorem \ref{teorema}, for any Holder
potential $\phi $, it follows that there exists a subset $E
\subset \Lambda$, with $\mu_\phi(E) = 1$ and an infinite
subsequence $(n_k)_k$ such that for any $z \in E$ we have the weak
convergence of measures $$ \mu^z_{n_k} \mathop{\to}\limits_{k \to
\infty} \mu_\phi $$ In particular, if $\mu_0$ is the measure of
maximal entropy, it follows that for $\mu_0$-almost all
 points $x \in \Lambda$,
$
\frac {1}{n_k} \mathop{\sum}\limits_{y \in f^{-n_k} (x) \cap \Lambda}
\frac{\mathop{\sum}\limits_{i = 0}^{n_k-1} \delta_{f^i y}}
{\text{Card}(f^{-n_k} (x) \cap \Lambda)} \mathop{\longrightarrow}\limits_{k \to
\infty} \mu_0$, for a subsequence $(n_k)_k$.
\end{cor}

\begin{proof}
Let us fix $g \in \mathcal{C}(\Lambda, \mathbb R)$. From the
convergence in $\mu_\phi$-measure of the sequence of functions $z
\to \mu^z_n(g), n \ge 1$ obtained from Theorem \ref{teorema}, it
follows that there exists a borelian set $E(g)$ with
$\mu_\phi(E(g)) = 1$ and a subsequence $(n_p)_p$ so that
$\mu^z_{n_p}(g) \mathop{\to}\limits_p \mu_\phi(g), z \in E(g)$.

Let us consider now a sequence of functions $(g_m)_m$ dense in
$\mathcal{C}(\Lambda, \mathbb R)$. By applying a diagonal sequence
procedure we shall obtain then a subsequence $(n_k)_k$ so that
$\mu^z_{n_k}(g) \mathop{\to}\limits_k \mu_\phi(g), \forall z \in
\mathop{\cap}\limits_m E(g_m)$. We notice also that
$\mu_\phi(\mathop{\cap}\limits_m E(g_m)) = 1$, since
$\mu_\phi(E(g_m)) = 1, m \ge 1$. \  But since any continuous
function $g$ can be approximated in the uniform norm by a function
$g_m$, it will follow that $\mu^z_{n_k}(g) \mathop{\to}\limits_k
\mu_\phi(g), \forall z \in E:= \mathop{\cap}\limits_m E(g_m)$.
Therefore we obtain that $\mu^z_{n_k} \mathop{\to}\limits_k
\mu_\phi, z \in E$, i.e we have weak convergence of the measures
$\mu^z_{n_k}, k \ge 1$, on a set of $z$ having full
$\mu_\phi$-measure in $\Lambda$.

\end{proof}

The Theorem applies to  \textbf{Anosov endomorphisms} in
particular.

\begin{cor}\label{Anosov}

Assume that $f:M \to M$ is an Anosov endomorphism without critical
points on a Riemannian manifold. Let also $\phi$ a Holder
continuous potential on $M$ and $\mu_\phi$ the equilibrium measure
of $\phi$. Then $$ \int_M |< \frac 1n \mathop{\sum}\limits_{y \in
f^{-n} (x)\cap \Lambda} \frac{e^{S_n\phi(y)}}{\mathop{\sum}\limits_{z \in
f^{-n}(x)\cap \Lambda} e^{S_n\phi(z)}} \cdot \mathop{\sum}\limits_{i=0}^{n-1}
\delta_{f^i y} - \mu_\phi, g >| d\mu_\phi(x)
\mathop{\to}\limits_{n \to \infty} 0 , \forall g \in
\mathcal{C}(M, \mathbb R) $$

\end{cor}

We can compare these results to the usual SRB measure for the
endomorphism $f$, defined as a measure $\mu$ having the property
that for any measurable partition $\eta$ of $\hat M$ subordinate
to the lifts of the local unstable manifolds, and for $\hat \mu$
almost all $\hat x \in \hat M$, the projection of the conditional
measure of $\hat\mu$, namely $\pi_* (\hat \mu^\eta_{\hat x})$ is
absolutely continuous with respect to the induced Lebesgue measure
on $W^u_{\hat x}$ (\cite{L}, \cite{QS}). In \cite{QS} it is shown
that $\mu$ satisfies the SRB property for the Anosov endomorphism
$f$ if and only if $\frac 1n
\mathop{\sum}\limits_{i=0}^{n-1}\delta_{f^i x}
\mathop{\to}\limits_n \mu$ for Lebesgue almost every $x\in M$.

\begin{cor}\label{inv}
Let $f :M \to M$ be an Anosov endomorphism, $\phi:\Lambda \to
\mathbb R$ a Holder potential and assume that the equilibrium
measure $\mu_\phi$ is absolutely continuous with respect to the
Lebesgue measure on $M$. Then the measure $\mu_\phi$ with this
property is unique, it is an SRB measure and it also satisfies an
inverse SRB condition in the sense that there exists a set $E$ of
full Lebesgue measure in $M$ and a sequence $(n_k)_k$ such that
$\mu^z_{n_k} \mathop{\to}\limits_k \mu_\phi, z \in E$.
\end{cor}

\begin{proof}
The proof follows immediately from Theorem \ref{teorema},
Corollary \ref{Anosov} and from the results of \cite{L} and
\cite{QS}. The potential $\phi$ for which $\mu_\phi$ is SRB, can
be taken in fact to be the unstable potential.

\end{proof}

A classical example of Anosov endomorphism without critical points
is a \textbf{toral hyperbolic endomorphism} $f_A:\mathbb T^m \to
\mathbb T^m$, associated to an $m \times m$ integer valued matrix
$A$ whose eigenvalues $\lambda_i$ all have absolute values
different from 1, and whose determinant $\text{det}(A)$ is not
necessarily 1 in absolute value. Each point from $\mathbb T^m$ has
exactly $|\text{det} A|$ preimages in $\mathbb T^m$. If we
consider the potential $\phi \equiv 0$, then the equilibrium
measure of $\phi$ is the Haar measure $\omega$ on $\mathbb T^m$
which is also the measure of maximal entropy (its entropy is equal
to $\mathop{\sum} \limits_{\lambda_i, |\lambda_i|
>1} \log |\lambda_i|$, where each eigenvalue is taken with its
multiplicity). In \cite{Wa} it was proved the asymptotic
distribution of periodic points towards $\omega$; here we prove
the convergence towards $\omega$, which is also the unique measure
of maximal entropy,  of the measures $\mu^x_n, n $ corresponding
to the potential $\phi \equiv 0$, for $\omega$-almost all points
$x \in \mathbb T^m$. We thus obtain the existence of an
\textit{inverse SRB measure} in this case.

Moreover Theorem \ref{teorema} applies also to smooth (say
$\mathcal{C}^2$) perturbations $f_{A, \vp}$, of hyperbolic toral
endomorphisms $f_A$. Indeed they will also be hyperbolic on the
$m$-dimensional torus $\mathbb T^m$ and the basic set considered
is the whole $\mathbb T^m$. Also the non-invertible map $f_{A,
\vp}$ remains $|\text{det}(A)|$-to-1 on $\mathbb T^m$. We thus
obtain the weighted distribution (with respect to a Holder
potential $\phi$ on $\mathbb T^m$) of preimage sets of $f_{A,
\vp}$, with respect to the equilibrium measure $\mu_\phi$ of
$\phi$, for \textbf{perturbations of hyperbolic toral
endomorphisms}.

Anosov endomorphisms on \textbf{infranilmanifolds} represent a
generalization of toral linear endomorphisms (see the Remark at
the end of Section 1). Let us notice that our Theorem
\ref{teorema} applies to Anosov endomorphisms on infranilmanifolds
which are not topologically conjugate to Anosov diffeomorphisms
nor to expanding maps. Thus, besides Theorem \ref{teorema}, one
cannot apply any of the previously known results for the
distributions of preimages from the case of diffeomorphisms
(\cite{Bo}), or expanding endomorphisms (\cite{Ru}).

Theorem \ref{teorema} applies also to \textbf{hyperbolic basic
sets of saddle type} for endomorphisms which \textit{are not
necessarily Anosov}, like the class of examples from \cite{MU},
namely skew products with overlaps in their fibers $F:X \times V
\to X \times V, F(x, y) = (f(x), h(x, y))$, where $f: X \to X$ is
an expanding map on a compact metric space, while $h(x, \cdot) :V
\to V$ (denoted also by $h_x$) is a contraction on an open convex
set $V \subset \mathbb R^m$; $h_x$ is assumed to depend
continuously on $x \in X$. The basic set is in this case given by
$$\Lambda:= \mathop{\cup}\limits_{x \in
X}\mathop{\cap}\limits_{n=0}^\infty \mathop{\cup}\limits_{z \in
f^{-n} x} h_z^n(\bar V),$$ where $h_z^n:= h_{f^{n-1}z} \circ
\ldots \circ h_z, n \ge 1, z \in X$. In \cite{MU} we studied the
conditional measures of equilibrium states induced on fibers and
their relation to the stable dimension of fibers. So from Theorem
\ref{teorema} we obtain the weighted distributions of preimages of
the non-invertible map $F$ over $\Lambda$, with respect to
equilibrium measures of Holder potentials. \

We can collect the above remarks in the following:

\begin{cor}\label{final}

a) The conclusions of Corollary \ref{Anosov} hold in particular
for toral hyperbolic endomorphisms and for smooth perturbations of
these.

b) The conclusions of Theorem \ref{teorema} hold for the basic
sets of hyperbolic skew products with overlaps in their fibers
from \cite{MU}, as well as for the attractors of the noninvertible
horseshoes from \cite{Bot}.

\end{cor}

\

\textbf{Acknowledgements:} This work was supported by
CNCSIS-UEFISCSU through Project PN II Idei-1191/2008.

\

\textbf{Email:}  Eugen.Mihailescu\@@imar.ro

Institute of Mathematics of the Romanian Academy, P. O. Box 1-764,

RO 014700, Bucharest, Romania.

Webpage: www.imar.ro/\~~mihailes

\end{document}